\documentclass{article}
\usepackage{graphicx} % Required for inserting images
\usepackage[utf8]{inputenc}
\usepackage[T1]{fontenc}
\usepackage[english]{babel}
\usepackage{geometry}
\usepackage{thmtools}
\usepackage{hyperref}
\usepackage{ragged2e}
\usepackage{textcomp}
\usepackage{amsmath}
\usepackage{amsfonts}
\usepackage{amsthm}
\usepackage{amssymb}
\usepackage{mathrsfs}
\usepackage{fancyhdr}
\usepackage{indentfirst}
\usepackage{tikz}
\usepackage{tikzscale}
\usepackage{stmaryrd}
\usepackage[all]{xy}
\usepackage{xcolor}
\usepackage{setspace}
\usepackage{multicol}

\newtheorem{teo}{Theorem}

\newtheorem{lem}[teo]{Lemma}
\newtheorem{prop}[teo]{Proposition}

\theoremstyle{definition}
\newtheorem{defn}[teo]{Definition}

\newtheorem{exem}{Example}
\newtheorem{obs}[teo]{Remark}
\newtheorem{theoremx}{Theorem}

\newcommand{\C}{\mathbb C}
\newcommand{\R}{\mathbb R}

\newcommand{\Z}{\mathbb Z}
\newcommand{\N}{\mathbb N}

\newcommand{\vet}[1]{\ensuremath{\overrightarrow{#1}}}
\newcommand{\dsp}{\displaystyle}

\newcommand{\lip}{\langle}
\newcommand{\rip}{\rangle}
\newcommand{\Ssig}{S_\sigma}

\newcommand{\sigduo}{\check{\sigma}}
\newcommand{\ord}{\mbox{ord}\;}
\title{Whitney equisingularity for families of hypersurfaces in toric varieties}
\author{Tha\'is Maria Dalbelo, Danilo da N\'obrega Santos}
\date{\it{Dedicated to the memory of Maria Aparecida Soares Ruas}}

\begin{document}

\maketitle
\begin{abstract}
In this paper, we establish conditions for a family $\{f_t\}$ of functions, with not necessarily isolated singularities, defined on a toric variety so that the associated family of hypersurfaces $\{f_t^{-1}(0)\}$ is Whitney equisingular. We work in the setting of toric varieties with arbitrary singular sets. This extends previous results by Eyral and Oka concerning families $\{F_t\}$ of functions in $\mathbb{C}^n$, with not necessarily isolated singularities, ensuring that the corresponding hypersurface family $\{F_t^{-1}(0)\}$ is Whitney equisingular.
\end{abstract}
\justifying
\section{Introduction}

The concept of Whitney equisingularity plays a crucial role in Algebraic Geometry and Singularity Theory. It provides a rigorous framework to study families of varieties with singular points by ensuring that certain geometric and topological properties remain consistent across the family. The search for invariants and conditions to describe Whitney equisingularity in families of varieties is
one of the main questions in Singularity Theory and has been studied by many authors, see for instance \cite{DamonGaffney, Gaffney, Gaffney2, BBT4, Tei1}. 

%This stability is particularly significant for understanding the behavior of singularities under deformation, as it guarantees that the families exhibit predictable and manageable characteristics. Whitney equisingularity is essential for applications in complex geometry, differential topology, and even mathematical physics and has been studied by many authors, see for instance \cite{Gaffney, Gaffney2, BBT4, Tei1}.

%Although of significant importance, given a family, it is generally quite challenging to ascertain whether such a family is Whitney equisingular. 

One concept strongly associated with Whitney equisingularity of families of varieties is the concept of Newton polyhedron. On unpublished notes Brian\c con studied the Whitney equisingularity for a family of Newton non-degenerate isolated hypersurface singularities. More precisely, let $F(t,z) = F(t,z_1,\dots,z_r)$ be a family of non-constant polynomial functions on $\mathbb{C} \times \mathbb{C}^r$, satisfying $F(t,0) = 0$ for all sufficiently small values of $t \neq 0$. We denote $F_t(z) = F(t,z)$, and associate to the family $\{F_t\}$ of functions the corresponding family of hypersurfaces $\{V(F_t)\} = \{F_t^{-1}(0)\}$ in $\mathbb{C}^r$. With this notation, J. Briançon \cite{briancon} proved the following result.

\begin{theoremx}
    Suppose that for all $t$ sufficiently small, the following conditions are satisfied:
    \begin{enumerate}
        \item $F_t$ has an isolated singularity at $0\in\C^r$;
        \item the Newton boundary $\Gamma(F_t;z)$ of $F_t$ at $0$ is independent of $t$;
        \item $F_t$ is non-degenerate (in the sense of \cite{kouchnirenko1976,oka1979bifurcation}).
    \end{enumerate}
    Then the family of hypersurfaces $\{V(F_t)\}$ is Whitney equisingular.
\end{theoremx}

In \cite{eyraloka1}, Eyral and Oka gave a generalization of Brian\c{c}on's result to families of hypersurfaces with non-isolated singularities. To do this, they introduced the concept of \emph{admissible family}. Roughly speaking, a family $\{F_t\}$ is admissible if
the Newton boundary is constant (with respect to $t$) and if the family $\{F_t\}$ is a family of non-degenerate hypersurfaces satisfying the ``uniform local tameness'' condition (see \cite[Definition 3.7]{eyraloka1}). In  \cite{eyraloka2} they generalize this result to families of complete intersection varieties defined in $\mathbb{C}^r$.

These results are formulated in the ambient space $\mathbb{C}^r$, whose toric structure plays a subtle but fundamental role in the analysis. Indeed, $\mathbb{C}^r$ is a very special toric variety. Often, the combinatorics of the semigroup that defines $\mathbb{C}^r$ is so naturally employed in the study of various objects within $\mathbb{C}^r$ that we may not even realize the toric structure is being utilized. This semigroup is generated by the canonical basis of $\mathbb{R}^r$.
Since the class of toric varieties includes elements with arbitrary singular sets, developing studies similar to those conducted in $\mathbb{C}^r$ requires understanding the obstacles to extending such results to the singular case.

 %In this context Matsui and Takeuchi presented in \cite{matsui2011} the concept of non-degenerate function defined in a toric variety $X$. To do this they used the combinatorics provided by the semigroup which generate $X$. 

The main goal of this work is to generalize the results of Eyral and Oka to the context of functions defined on a toric variety, by using the notion of non-degenerate functions on arbitrary toric varieties, as introduced by Matsui and Takeuchi in \cite{matsui2011}, together with the combinatorics that resides in these varieties and arises from the semigroups that generate them. 

%More precisely, consider the toric variety $X(\Ssig)\subset\mathbb{C}^r$ defined by the strongly convex cone $\sigma\subset\mathbb{R}^n$ with maximal dimension as in \cite{brasselet,fulton}. Let $(t,z):=(t,z_1,\dots, z_r)$ be coordinates of the set $ \C\times X(\Ssig)$, let $0\in U\subset X(\Ssig)$ be an open set, $0\in D\subset\C$ be an open disc, and
%$$
%\begin{array}{cccc}
 %   f: & (D\times U, D\times\{0\}) & \longrightarrow & (\C,0)  \\
 %      &       (t,z)               &    \longmapsto  & f(t,z)
%\end{array}
%$$
%be a polynomial function satisfying $f(t,0)=0$, for all $t \in D$. We write $f_t(z)=f(t,z)$ and denote by $V(f_t)\subset U$ the hypersurface defined by $f_t(z)$. Our focus is on studying the local structure of the singular loci of these hypersurfaces at the origin $0\in X(\Ssig)$ as the parameter $t$ approaches $0$. 

As in \cite{eyraloka1}, we will address non-isolated singularities. In this context, it is necessary to account not only for the compact faces of the Newton polyhedron but also for an additional class of faces.

Following the ideas presented in \cite{oka1990}, we define the notions of essential non-compact face and non-compact Newton boundary in the context of toric varieties. We also develop a tool to address the essential non-compact faces. Applying a similar approach to that in \cite{oka2015mixed}, we introduce the concept of local tameness and, consequently, adapt the admissibility condition (originally introduced in \cite{eyraloka1}) to the toric setting. With these notions established, we are able to prove the following theorem.

\begin{theoremx}[see Theorem \ref{mainteo}]
   Given an admissible family of polynomial functions $\{f_t\}$ on a toric variety $X$, then the associated family of hypersurfaces $\{V(f_t)\} \subset X$ is Whitney equisingular.
\end{theoremx}

%\textcolor{blue}{pensie em aldo tipo: In addition to encompassing results obtained for $\C^r$  that allow situations involving singularities in the domain of the functions under consideration}. 
%\textcolor{red}{Besides generalizing the results obtained in $\mathbb{C}^r$ to situations involving singularities}, toric varieties form a profoundly rich class in algebraic geometry, offering a combinatorial framework that enables explicit constructions and deep insights into the structure and behavior of singular spaces. Moreover, there are important classes of varieties intersecting with the class of toric varieties. Among them, we can quote the class of determinantal varieties (see Example \ref{exdeterminantal}). 

This paper is structured as follows. In Section $2$, we define toric varieties and present some important notions related to them. We also adapt to the toric case fundamental tools that were used by Eyral and Oka in \cite{eyraloka1}. In Section $3$, we introduce the concept of essential non-compact faces and define the condition of local tameness for the toric case. We establish the notion of admissible family, which combines non-degeneracy, tameness, and stability of the Newton boundary. Furthermore, we prove that under the condition of Newton boundary invariance, a family of non-degenerate functions possesses good geometric properties, such as smoothness and the existence of a Milnor ball. Finally, in Section 4, we present and prove the main theorem, which guarantees that admissible families of functions on toric varieties are Whitney equisingular. We conclude this work by presenting an example in which the family of hypersurfaces is defined on a toric surface that is also determinantal, thereby highlighting the connection between toric geometry and determinantal structures.

 %When $F_t$ is the product of functions $G_t$ and $H_t$ such that $V(G_t)\cap V(H_t)$ has dimension greater than $1$ at $0$, then $F_t$ is never non-degenerate (for more details see \cite{kouchnirenko1976}). In \cite{eyraloka2} the authors extend the concept of the admissibility to a family of functions of type $F_t(z)=F^1_t(z)\cdots F^{k_0}_t(z)$ in order to solve this limitation. However, we will see that in the toric case this is not necessarily a limitation.

\section{Generalities: toric varieties and non-degenerate conditions}

For the convenience of the reader and to clarify some notations, in this section we adapt key concepts and results presented by Eyral and Oka \cite{eyraloka1}. Specifically, we reformulate these results in the setting of polynomial functions defined on toric varieties, using the geometric and combinatorial structures that arise in this context. These adaptations will be essential for the development of this work.
For further details on toric varieties and non-degeneracy conditions, we refer to \cite{brasselet,matsui2011,okabook}.

\smallskip

\subsection{Toric varieties}Let $\sigma \subset \mathbb{R}^n$ be a strongly convex rational polyhedral cone with maximal dimension $n$ and let  
\[
\check{\sigma} = \left\{v \in \mathbb{R}^n; \ 
\ \left\langle u,v\right\rangle \geq 0 \ \ \text{for any} \ \ u \in 
\sigma\right\},
\] be
the dual cone of $\sigma$, where $\langle \cdot , \cdot \rangle$ is the usual inner product in $\R^n$.
Then the dimension of $\check{\sigma}$ is $n$ and we obtain a semigroup  $S_\sigma = \check{\sigma}\cap \Z^n$ which is finitely generated. We denote by  $\{b_1,\dots,b_r\}$ an ordered generator system of 
$S_{\sigma}$.

\begin{defn}
	The $n$-dimensional affine toric variety (or simply toric variety) $X(S_\sigma) \subset \mathbb{C}^r$ is defined by the 
	spectrum of $\mathbb{C}[S_{\sigma}]$, i.e 	$X(S_\sigma)=\rm{Spec}(\mathbb{C}[S_{\sigma}])$, where $\rm{Spec}(\mathbb{C}[S_{\sigma}])$ denotes the set of maximal ideals in $\mathbb{C}[S_{\sigma}]$.
\end{defn}
 
%Let $\sigma\subset\R^n$ be a strongly convex lattice cone with maximal dimension, $\check{\sigma}\subset(\R^n)^*$ its dual cone, $S_\sigma=\lip b_1,\dots,b_r\rip$ the finitely generated semigroup given by $\check{\sigma}\cap \Z^n$, and let $X(S_\sigma)\subset \mathbb{C}^r$ be the affine toric variety associated with the cone $\sigma$.

There exists an action from the algebraic torus $(\C^*)^n$ to $X(\Ssig)$ and from our assumptions on $\sigma$ this action has a unique zero-dimensional orbit, which is the origin $0\in\C^r$. Furthermore, we may assume $\check{\sigma} \subset \mathbb{R}^n_+$, where $\mathbb{R}^n_+$ is the first orthant of $\mathbb{R}^n$.

Note that the semigroup $\Z^n$ and the ring $\C[\Z^n]$ are generated by $\{\pm e_i\}$ and $\{x_i,x_i^{-1}\}$, respectively, where $i=1,\dots,n$ and $\{e_i\}$ denote the canonical basis of $\mathbb{R}^n$. Recall the identifications:
\[
(\C^*)^n=\mbox{Spec}(\C[\Z^n])\cong\mbox{Hom}_{\mbox{sg}}(\Z^n,\C),
\]
where $\mbox{Hom}_{\mbox{sg}}(\Z^n,\C)$ denotes the set of semigroup homomorphisms from $\Z^n$ to $\C$. And, let us also remember the identifications:
\[X(\Ssig)=\mbox{Spec}(\C[\Ssig])\cong\mbox{Hom}_{\mbox{sg}}(\Ssig,\C),\]
in which each $\psi\in\mbox{Hom}_{\mbox{sg}}(\Ssig,\C)$ is associated with a point $x\in X(\Ssig)$ given by $x=(\psi(b_1),\dots,\psi(b_r))$.

\begin{defn}
    Let $\sigma \subset \mathbb{R}^n$ be a strongly convex rational polyhedral cone, and let $\Ssig$ and $X(\Ssig)$ be as above.  For each face $\gamma$ of $\sigma$ we associate a distinguished point $x_\gamma=(\psi_\gamma(b_1),\dots,\psi_\gamma(b_r))\in X(\Ssig)$, where $\psi_\gamma\in\mbox{Hom}_{\mbox{sg}}(\Ssig,\C)$ is given by
    \[
	\psi_\gamma(b_i)=
	\begin{cases}
		1, & \mbox{ if }\;\;\;b_i\in\gamma^\perp\\
		0, & \mbox{ otherwise}
	\end{cases},
    \]
    and $b_i \in \gamma^\perp$ means that $\lip b_i, u\rip = 0$ for all $u \in \gamma$.
\end{defn}

Distinguished points belong to $X(S_{\sigma})$, then it makes sense to consider their orbit under the action.

\begin{defn}\label{orbt}
	Let $\sigma\subset\R^n$ be a strongly convex rational polyhedral cone and $\gamma$ a face of $\sigma$. The orbit of the face  $\gamma$, under the action, is the orbit of its corresponding distinguished point $x_{\gamma}$, and is denoted by $\mathcal{O}_{\gamma}$.
\end{defn}

The following result (see \cite[Theorem  $3.2.6$]{cox2024}) describes the closure of an orbit.

\begin{teo} \label{decorb}
	Let $\sigma\subset\R^n$ be a strongly convex rational polyhedral cone and $X(S_{\sigma})$ the toric variety generated by $S_{\sigma}$. Then, we have
	\begin{itemize}
		\item [i.] $X(S_{\sigma})=\dsp\bigsqcup_{\gamma\prec\sigma}\mathcal{O}_\gamma$, where $\gamma \prec \sigma$ denotes a face $\gamma$ of $\sigma$ (including $\sigma$ itself);
        
		\item [ii.] If $\gamma\prec\sigma$ and $\overline{\mathcal{O}}_\gamma$ denotes the closure of the orbit $\mathcal{O}_\gamma$, then $\overline{\mathcal{O}}_\gamma=\dsp\bigsqcup_{\gamma\prec\sigma}\mathcal{O}_\sigma$.
	\end{itemize}
\end{teo}

The decomposition given in item $i.$ is a Whitney stratification of $X(S_{\sigma})$. Moreover, there exists a one-to-one correspondence between the faces of $\sigma$ and the faces of its dual cone $\check{\sigma}$ (see \cite[Chapter 1, \S 1.2]{cox2024}). We denote by $x_{\tau}$ the distinguished point of the face $\gamma$ of $\sigma$ which corresponds to the face $\tau$ of $ \check{\sigma}$ and we will simply call it the distinguished point of $\tau$.

Let $\tau = \langle b_{i_1}, \dots, b_{i_m} \rangle$ be a face of $\check{\sigma}$ and denote by $I_{\tau} = \{i_1, \dots, i_m\} \subset \{1, \dots, r\}$.
 We define the set
\begin{eqnarray*}
	X(\Ssig)^{I_{\tau}}:= \overline{\mathcal{O}}_{x_{\tau}},\;\mbox{\ \ where}\;x_i=0\;\mbox{if}\; i\notin I_{\tau},
    \end{eqnarray*}
 $x_{\tau} = (x_1,\dots,x_r)\in X(\Ssig)$ is the distinguished point of the face $\tau$, and $ \mathcal{O}_{x_{\tau}}$ denotes the orbit of the point $x_{\tau}$. Note that $X(\Ssig)^{I_{\tau}}$ is a toric variety, since it is the closure of an orbit.    
Similarly, define
\begin{eqnarray*}	{X(\Ssig)^*}^{I_{\tau}}:=\mathcal{O}_{x_{\tau}},\;\mbox{\ \ where}\;x_i=0\;\mbox{if, and only if,}\; i\notin I_{\tau}.
\end{eqnarray*}
 Consequently, $X(\Ssig)^{\emptyset}=X(\Ssig)^*{}^\emptyset=\{0\}$ and ${X(\Ssig)^*}^{\{1,\dots,r\}}$ is the dense orbit homeomorphic to the algebraic torus $(\C^*)^n$. Moreover, if $U = (U_1, \dots, U_r)\in X(\Ssig)^*{}^{I_{\tau}}$, then $U_{i_1},\dots,U_{i_m}\in\C^*$  and all other coordinates are  zero.

We also remark that $X(\Ssig)^*{}^{I_{\tau}}$ is contained in the smooth locus of $X(\Ssig)$, since $X(\Ssig)^*{}^{I_{\tau}}$ is an orbit of the torus action. 
\begin{obs}\label{obsI}
 The subset $I_{\tau}$ is not arbitrary within $\{1,\dots, r\}$, it depends on $\check{\sigma}$. For example, consider the toric variety $X(\Ssig)\subset\C^4$ associated with the semigroup $\Ssig=\lip e_2+2e_3,2e_1+e_2,e_1+3e_3,e_1+e_2+e_3\rip$ (see Figure \ref{cone1}), in which $\sigma=\lip2e_1-4e_2+2e_3,3e_1+2e_2-e_3,-3e_1+6e_2+e_3\rip \subset \mathbb{R}^3$ and $\check{\sigma}=\lip e_2+2e_3,2e_1+e_2,e_1+3e_3\rip$. The possible subsets $I_{\tau}$ are $\emptyset$, $\{1,2,3,4\}$, $\{1,2,4\}$, $\{1,3\}$, $\{2,3\}$, $\{1\}$, $\{2\}$, and $\{3\}$. 
 For instance, the set $\{1,4\}$ cannot be a $I_{\tau}$, since $\lip e_2+2e_3,e_1+e_2+e_3\rip$ is not a face of $\check{\sigma}$.
 \end{obs}
 \begin{figure}[!thb]
     \centering
     \includegraphics[width=3.5cm]{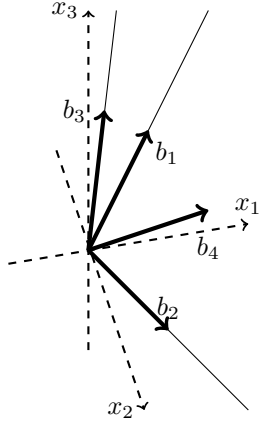}
     \caption{Polyhedral cone in $\mathbb{R}^3$}
     \label{cone1}
 \end{figure}

\subsection{Functions on toric varieties and non-degeneracy conditions} From now on, whenever we write $X(\Ssig)^*{}^I$ or $X(\Ssig){}^I$, we assume that the subset $I$ is given by $I_{\tau} = \{i_1, \dots, i_m\} \subset \{1, \dots, r\}$, for some face $\tau = \langle b_{i_1}, \dots, b_{i_m} \rangle$ of $\check{\sigma}$.

Given an $n$-dimensional toric variety $X(\Ssig)\subset \C^r$, there is a \emph{canonical embedding} of $(\C^*)^n$ into $X(\Ssig)$
\begin{equation}\label{canonicalh}
\begin{array}{cccc}
	h:  & (\C^*)^n & \longrightarrow & X(\Ssig)\\
	    &  \xi     &  \longmapsto    & (\xi^{b_1},\dots,\xi^{b_r}),
\end{array}
\end{equation}
where $\xi=(\xi_1,\dots,\xi_n)\in(\C^*)^n$, $
\xi^{b_i}=\xi_1^{\beta_i^1}\cdots\xi_n^{\beta_i^n}$ and $b_i=(\beta_i^1,\dots,\beta_i^n)\in \{b_1,\dots,b_r\}$. Let $$g(z)=\dsp\sum_{\Lambda\in\N^r\setminus\{0\}}a_\Lambda z^\Lambda$$ be a polynomial function on $X(\Ssig)$, i.e $g$ is the restriction of a polynomial function $G: \C^r\longrightarrow\C$ to the toric variety $ X(\Ssig)$. We define the polynomial function $L_g$ from $(\C^*)^n$ to $\C$, as follows
$$
 L_g(\xi):=(g\circ h)(\xi)=\dsp\sum_{\lambda\in \Ssig}{a_\Lambda}\xi^\lambda,
$$
where $\lambda=\dsp\sum_{i=1}^r \Lambda_ib_i$, and  $\Lambda=(\Lambda_1,\dots,\Lambda_r)\in\N^r\setminus \{0\}$. We call $\Lambda$ the \emph{representation} of $\lambda$ in $\N^r$. 
%Note that the polynomial $g$ is nothing more than the restriction on $X(\Ssig)$ of the polynomial $G:\mathbb{C}^r\longrightarrow\mathbb{C}$. In other words $g=G\rvert_{X(\Ssig)}$. 

\begin{defn}
The convex hull of $\dsp\bigcup_{\lambda\in supp(L_g)}(\lambda+ \check{\sigma}) \subset \check{\sigma}$ is called the \emph{Newton polyhedron} of $g$ with respect to the coordinates $z$ and will be denoted by $\Gamma_+(g;z)$, where $supp(L_g):=\{\lambda\in\Z^n:a_\Lambda\neq0\}$.    
\end{defn}
%Given $w\in\sigma$, the inner product $\lip w,x\rip$ for $x\in\Gamma_+(g;z)$, reaches a minimal value, since $\Gamma_+(g;z)\subset\check{\sigma}$ \cite{matsui2011}. Denoting by $d_w=\dsp\min_{x\in\Gamma_+(g;z)}\lip w,x\rip$, the \emph{face} $\Delta_w$ of $\Gamma_+(g;z)$ is defined as the set where $\lip w,x\rip$ takes its minimal value on $\Gamma_+(g;z)$, i.e.
%	$$
%	\Delta_w=\{x\in\Gamma_+(g;z):\lip w,x\rip=d_w\}.
%	$$
%We call $w$ a \emph{weight vector}.

Let $w \in \sigma \cap \mathbb{Z}^n $. The inner product $\lip w, x \rip$, for $x \in \Gamma_+(g; z)$, attains its minimum value because $\Gamma_+(g; z) \subset \check{\sigma}$ (see \cite{matsui2011} pg $119$). Let $d_w = \dsp\min_{x \in \Gamma_+(g; z)} \lip w, x \rip$. The \emph{face} $\Delta_w$ of $\Gamma_+(g; z)$ is then defined as the set where $\lip w, x \rip$ reaches this minimum on $\Gamma_+(g; z)$, that is,
$$ \Delta_w = \{x \in \Gamma_+(g; z) : \lip w, x \rip = d_w\}. $$
We refer to $w$ as a \emph{weight vector}.

%The \emph{compact Newton boundary} of $g$ is the union of all compacts faces of $\Gamma_+(g;z)$ and it will be denoted by $\Gamma(g;z)$. Note that if $\lip b_{i_1},\dots,b_{i_m}\rip$ generates an  variety $X(\Ssig)^I$ with $I=\{i_1,\dots,i_m\}$ and $w\in\sigma$ is such that $\lip w,b_{i_j}\rip=0$ for any $1\leq j\leq m$, then $\Delta_w$ is a non-compact face contained, up to translation, in the smallest linear subspace generated by $\lip b_{i_1},\dots,b_{i_m}\rip$. Furthermore, if $\lip w, b_i\rip\neq0$ for any $1\leq i\leq r$, then  $\Delta_w$ is a compact face of $\Gamma_+(g;z)$. 

Let $\tau$ be the face of $\check{\sigma}$ generated by $\lip b_{i_1}, \dots, b_{i_m} \rip$, and $w \in \sigma$ satisfying $\lip w, b_{i_j} \rip = 0$ for all $1 \leq j \leq m$, then $\Delta_w$ is a non-compact face of the Newton polyhedron $\Gamma_{+}(g;z)$. On the other hand, if $\lip w, b_i \rip \neq 0$ for all $1 \leq i \leq r$, then $\Delta_w$ is a compact face of $\Gamma_+(g; z)$. The \emph{compact Newton boundary} of $g$ is defined as the union of all compact faces of $\Gamma_+(g; z)$ and is denoted by $\Gamma(g; z)$.

%In this case, the weight vector $w=\dsp\sum_{i=1}^{r}W_ib_i$ and its representation is $W=(W_1,\dots,W_r)\in(\N^*)^r$. 

Let $\Delta$ be a face of $\Gamma_+(g;z)$. The \emph{face function} $g_\Delta$ is defined by
$$
g_\Delta(z)=\sum_{\lambda\in\Delta\cap\Ssig}a_\Lambda z^\Lambda
$$
where $\Lambda=(\Lambda_1,\dots,\Lambda_r)$ is the representation of $\lambda$ in $\N^r$. Therefore, we can consider the polynomial function $(L_g)_{\Delta}$ on $(\mathbb{C}^*)^n$ associated with the face function $g_\Delta$, given by
$$
(L_g)_{\Delta}(\xi):=(g_\Delta\circ h)(\xi)=\sum_{\lambda\in\Delta\cap \Ssig}a_\Lambda\xi^\lambda.
$$

%Here, we considering $\lip b_1,\dots,b_r\rip$ ordered basis of $S_\sigma$.
\begin{exem}\label{example1}
 Let $X(\Ssig)\subset\C^3$ be the toric variety determined by $\sigma=\lip e_2,2e_1-e_2\rip$, then $\check{\sigma}=\lip e_1,e_1+2e_2\rip$ and $\Ssig=\lip e_1,e_1+e_2,e_1+2e_2\rip \subset \mathbb{Z}^2$.
 Given the polynomial function     
 $$g(z_1,z_2,z_3)=z_1^4+z_1^2z_2+z_1z_2^2-z_1z_2z_3^2\;\;\;\Rightarrow\;\;\; L_g(\xi_1,\xi_2)=\xi_1^4+\xi_1^3\xi_2+\xi_1^3\xi_2^2-\xi_1^4\xi_2^5,
 $$
    the compact faces of $\Gamma_{+}(g;z)$ correspond to the line segments $\overline{AB}$, $\overline{BC}$, and $\overline{CD}$  where $A=(4,0)$, $B=(3,1)$, $C=(3,2)$ and $D=(4,5)$ (see Figure \ref{poli2}). On the other hand, the polynomial function $L_g$ on $\mathbb{C}^2$ has one compact face corresponding to the line segment $\overline{AB}$ (see Figure \ref{poli1}). 
    %In general, it is possible see that compact faces of a polynomial function $L(g)$ on $\mathbb{C}^n$ are also compact faces $g$ on $X(S_\sigma)$.
\end{exem}

\begin{obs}\label{obsfc} We observe that, in general, the compact faces of the polynomial function $L_g$ on $\mathbb{C}^n$ also correspond to compact faces of $g$ on $X(S_\sigma)$.

\begin{figure}[!thb]
    \begin{minipage}{0.45\linewidth}
    \centering
        \includegraphics[width=3cm]{polyex11.tikz}
        \caption{Newton polyhedron on $\C^2$}
        \label{poli1}
    \end{minipage}
    \begin{minipage}{0.45\linewidth}
    \centering
        \includegraphics[width=2.75cm]{polyex12.tikz}
    \caption{Newton polyhedron on $X(\Ssig)$ }
    \label{poli2}
    \end{minipage}
\end{figure}
\end{obs}

\begin{defn}\label{nondeg}
    The polynomial function $g$ on $X(\Ssig)$ is said to be \emph{non-degenerate} if, for any compact face $\Delta\subset\Gamma(g;z)$, the hypersurface $(L_g)_\Delta^{-1}(0)$ has no singular points in $(\C^*)^n$. In other words, the Jacobian matrix $J(L_g)_\Delta(\xi)\neq0$ for all $\xi\in (L_g)_\Delta^{-1}(0)\cap(\C^*)^n$.
\end{defn}

\begin{obs}\label{obsng}
    Let $g$ be non-degenerate, and let $\Delta\subset\Gamma(g;z)$ be a compact face. According to the chain rule, $g_\Delta$ has no critical points on the dense orbit $X(\Ssig)^*{}^{\{1,\dots,r\}}$. %Note that there is an abuse of notation here, actually $\Delta$ is not a face of $\mathbf{G}$. However, any $\lambda \in \Delta$ has a representation $\Lambda \in \mathbb{R}^r$, and the notation $\mathbf{G}_{\Delta}$ represents the restriction of $\mathbf{G}$ to monomials associated with $\Lambda$. 
\end{obs}

\begin{exem}\label{nondegexem}
    Let $X(\Ssig)\subset\C^3$ be the toric variety given in Example \ref{example1}. Considering the polynomial function $g(z_1,z_2,z_3)=z_1^4+z_2^4z_3-z_2^2z_3^2$, we have 
    $$L_g(\xi_1,\xi_2)=\xi_1^4+\xi_1^5\xi_2^6-\xi_1^4\xi_2^6.$$ 
    The Newton polyhedron $\Gamma_+(g;z)$ has three compact faces: the points $A=(4,0)$ and $B=(4,6)$, and the line segment $\Delta=\overline{AB}$, connecting $A$ to $B$. The polynomial function $(L_g)_\Delta$ associated with the face function $g_\Delta$ is $(L_g)_\Delta(\xi_1,\xi_2)=\xi_1^4-\xi_1^4\xi_2^6$, which has no critical points on $(\C^*)^2$. The same holds for $g_A$ and $g_B$. Therefore, $g$ is non-degenerate on $X(\Ssig)$. However, considering the cone $\sigma=\lip e_1,e_2,e_3\rip$, then $S_{\sigma}=\lip e_1,e_2,e_3\rip \subset \mathbb{Z}^3$, and $X(S_{\sigma})=\C^3$. Now taking the same polynomial function $G(z_1,z_2,z_3)=z_1^4+z_2z_3^4-z_2^2z_3^2$, we have $L_G=G$ and $G$ is degenerate. Indeed, $(L_G)_{\Delta'}(\xi_1,\xi_2,\xi_3)=\xi_2\xi_3^4-\xi_2^2\xi_3^2$ has critical points in $(\C^*)^3$, where $\Delta'=\overline{CD}$ is  
    the compact face  which is the line segment connecting the points $C=(0,1,4)$ and $D=(0,2,2)$.
\end{exem}

%\begin{obs}
 %   However, $G(z_1,z_2,z_3)=z_1^4+z_2z_3^4-z_2^2z_3^2$ is degenerate on $\C^3$ in the sense of \cite{eyraloka1}. Indeed, for the compact face $\Delta'=\overline{CD}$, where $C=(0,1,4)$ and $D=(0,2,2)$, the function $G^{\Delta'}(z_1,z_2,z_3)=z_2z_3^4-z_2^2z_3^2$ has critical points on $(\C^*)^3$.
%\end{obs}

One immediate difference between studying polynomial functions on $\mathbb{C}^r$ and on an $n$-dimensional toric variety \(X(S_{\sigma}) \subset \mathbb{C}^r\) lies in the ambient space where the Newton polyhedron is constructed. In the case of functions on \( \mathbb{C}^r \), the Newton polyhedron is contained in \( \mathbb{R}^r_+ \), whereas in the case of \( X(S_{\sigma}) \), it lies in \( \mathbb{R}^n_+ \). In this sense, the following lemma shows that a specific inner product is preserved, even across these different ambient settings.

\begin{lem}\label{pesos}
    Let $g$ be a polynomial function on a toric variety $X(\Ssig)\subset\C^r$, and let $w=(w_1,\dots,w_n)$ be a vector in $\sigma$ such that $w_i>0$ for $1\leq i\leq n$. If  $\Delta_w\subset\Gamma(g;z)$ is the compact face associated with $w$, then for any $\lambda \in \Delta_w \cap S_{\sigma}$ we have $\lip W,\Lambda\rip=\lip w,\lambda\rip$, where $\Lambda$ is the representation of $\lambda$ in $\mathbb{N}^r$ and $W=(\lip w,b_1\rip,\dots,\lip w,b_r\rip)$. 
    %Here, $\lip\cdot,\cdot\rip_n$ and $\lip\cdot,\cdot\rip_r$ denote the inner product on $\R^n$ and $\R^r$ respectively.
\end{lem}
\begin{proof}
    Since $\Lambda$ is the representation of $\lambda$ in $\mathbb{N}^r$, there exist $\Lambda_1,\dots,\Lambda_r\in\N$ such that $    \lambda=\displaystyle\sum_{i=1}^r\Lambda_ib_i$.
    On the other hand $$\sum_{i=1}^r\Lambda_ib_i=\left(\sum_{i=1}^{r}\Lambda_i\beta_1^{i},\dots,\sum_{i=1}^{r}\Lambda_i\beta_n^{i}\right).$$
    Thus, we conclude the result by writing $\Lambda=(\Lambda_1,\dots,\Lambda_r)$.\end{proof}

\subsection{Essential non-compact faces and locally tame functions}As mentioned in the introduction, this work addresses non-isolated singularities. To achieve this, it is essential to consider not only the compact faces of the Newton polyhedron but also an additional class of faces. Following the ideas presented in \cite{oka1990}, we introduce the concept of \emph{essential non-compact face} within the context of toric varieties. 
    
\begin{defn}\label{essenc}
	Let $g$ be a polynomial function on $X(\Ssig)$. A face $\Delta$ of $\Gamma_+(g;z)$ is an \emph{essential non-compact face} if there exists $w \in \sigma$ satisfying:
	\begin{enumerate}
		\item [(i)] $\Delta = \Delta_w$ and $g\rvert_{X(\Ssig)^{I_w}}\equiv0$, where $I_w:=\{i\in\{1,\dots,r\}:\lip w,b_i\rip=0\}$,
		\item [(ii)] for any $i\in I_w$ and any point $\alpha\in\Delta$, the half-line $\alpha+\R_+b_i\subset\Delta$.
	\end{enumerate}
	The set \( I_w \) does not depend on the choice of $w$.
 Therefore, it is called the \emph{non-compact direction} of $\Delta$ and is denoted by $I_\Delta$. The \emph{essential non-compact Newton boundary} of $g$ is defined as the union of $\Gamma(g;z)$ with the essential non-compact faces, which we denote by $\Gamma_{nc}(g;z)$.
\end{defn}

We denote by $\mathscr{I}_{nv}(g)$ (respectively, $\mathscr{I}_v(g)$) the collection of subsets $I\subset\{1,\dots,r\}$ such that $g\rvert_{{X(\Ssig)^I}}\not\equiv0$ (respectively, $g\rvert_{{X(\Ssig)^I}}\equiv0$). For any $I\in\mathscr{I}_{nv}$ (respectively, any $I\in\mathscr{I}_{v}$) the variety $X(\Ssig)^I$ is referred to as the \emph{non-vanishing} (respectively, \emph{vanishing}) variety.

\begin{exem}\label{goodex}
     Let $X(\Ssig)\subset\C^4$ be the toric variety generated by the semigroup $\Ssig=\lip e_2+2e_3,2e_1+e_2,e_1+3e_3,e_1+e_2+e_3\rip$ (see Remark \ref{obsI}). Consider the polynomial function   $$g(z)=z_1^2z_3^3+z_2^2z_3^3+z_3^4-5z_3^3z_4^3$$
     on $X(\Ssig)$. The non-compact face $\Delta=\overline{AB}+\R_+b_1+\R_+b_2+\R_+b_4$ is essential, where $\overline{AB}$ is the edge connecting the points $A = (3,2,13)$ and $B = (7,2,9)$. Indeed, choosing $w = (1,-2,1)$, which is orthogonal to $b_1$, $b_2$, and $b_4$, we obtain $\Delta = \Delta_w$ and $I_{\Delta} = \{1,2,4\}$, implying that $g_{|X(\Ssig)^{{\{1,2,4\}}}}\equiv0$. However, the non-compact face $\Theta=\overline{AC}+\R_+b_3$ (as well as $\Theta'=\overline{BC}+\R_+b_3$), where $C=(4,0,12)$ are not essential, since $X(\Ssig)^{\{3\}}$ is not a vanishing variety (See Figure \ref{fig1}). 
\end{exem}

\begin{figure}[!thb]
    \centering
    \includegraphics[width=6cm]{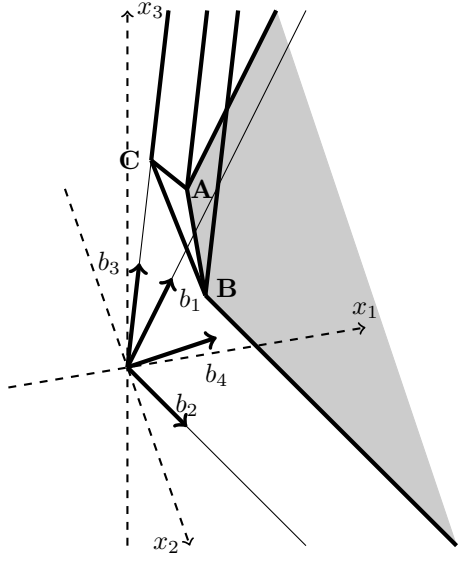}
    \caption{Essential non-compact face}
    \label{fig1}
\end{figure}

Let $X(\Ssig) \subset \mathbb{C}^r$ be the toric variety generated by the semigroup $\Ssig= \lip b_1,\cdots,b_r \rip$ and let $I\subset\{1,\dots,r\}$. We define the map
\begin{equation}\label{canonicalhI}
\begin{array}{cccc}
h^I: & (\C^*)^n & \longrightarrow & X(\Ssig)^I\\
     &  \xi     &   \longmapsto   & h^I(\xi),
\end{array}
\end{equation}
where $h_i(\xi)=0$ if $i\notin I$, and $h_i(\xi)=\xi^{b_i}$ otherwise. Consequently, we may also consider the restriction of the polynomial function $L_g$ to $X(\Ssig)^I$. In this case, we use the notation
$$
L_g^I(\xi):=(g_{|X(\Ssig)^I}\circ h^I)(\xi).
$$

\begin{prop}\label{propdeg}
    Let $g$ be a non-degenerate polynomial function on $X(\Ssig)\subset\C^r$. Then, for any $I\in\mathscr{I}_{nv}(g)$, the restriction $g\rvert_{X(\Ssig)^I}$ is non-degenerate as a function of variables $z_I:=\{z_i:i\in I\}$.
\end{prop}
\begin{proof}
    Let $\tau=\lip b_i\rip_{i\in I}$ be the face of the dual cone $\check{\sigma}$ generated by the vectors $b_i$, for $i\in I$. Let $\Delta$ be a compact face of the Newton polyhedron $\Gamma_+(g\rvert_{X(\Ssig)^I};z_I) \subset \Gamma_+(g;z)$, and let $w\in\sigma$ be a weight vector such that   
    $$\Delta_w = \{x \in \Gamma_+(g\rvert_{X(\Ssig)^I};z_I) : \lip w, x \rip = d_w\}=\Delta, \ \ \text{where} \ \ d_w = \dsp\min_{x \in \Gamma_+(g\rvert_{X(\Ssig)^I};z_I)} \lip w, x \rip.$$ If $R_{\check{\sigma}}=\lip v_1,\dots,v_N\rip$ is the semigroup associated with $\sigma$, then
    \[
    w=\sum_{i=1}^N\lambda_iv_i, \;\; \lambda_i\in\N\;\;\mbox{for}\;\;1\leq i\leq N.
    \]
    Now, consider $q=\dsp\sum_{i=1}^{N}q_iv_i$, where $q_i=\lambda_i$ for $i\in I$ and, $q_i=V\in\N^*$ for $i\not\in I$.
    Note that $V$ can be taken sufficiently large so that $\Delta_q=\Delta$. Thus,
    $$
    g_\Delta(z)=(g\rvert_{X(\Ssig)^I})_\Delta(z).
    $$
    Since $g$ is non-degenerate, it follows that $(L_g)_\Delta$ has no critical points in $(\C^*)^n$. As $g_\Delta$ contains only variables $z_i$, with $i\in I$, and  $X(\Ssig)^I \subset X(\Ssig)$ is a toric variety of dimension $n_I < n$,  $(L_g^I)_\Delta$ also has no critical points on $(\C^*)^{n_I}$.  
\end{proof}

%(escolhendo assim teremos fixado as entradas do vetor $w$ que garantem a ortogonalidade com a face $\tau\prec\check{\sigma}$ que contem a face $\Delta$ do poliedro, mantendo $\Delta$ como local de mínimo).
%\begin{tikzpicture}
%	\draw [ultra thick,->] (0,0.5)node[right] {$\mathbf{b_3}$} (0,0)--(0,1);
%	\draw [black!20, fill=black!20] (-1,-2)--(-1.5,-3.5)-- (3.42,-4.73)-- (6,3)--(3,3.5)--(-1,-2);
%	\draw [line width=2pt] (-1,-2)node[above, left] {\textbf{E}}--(3,3.5)--(0,5)node[left]{\textbf{F}}--(-1,-2);
%	\draw (3.35,3.75)node{\textbf{D}};
%	\draw [thick,dashed,->] (0,-3)--(0,6)node[right] {$x_3$}; 
%	\draw [thick] (0,0)--(0,5);
%	\draw [ultra thick,->](1,-0.25)node[above] {$\mathbf{b_1}$} (0,0)--(1.5,-0.25);
%	\draw [thick] (0,0)--(4.74,-0.79);
%	\draw [ultra thick,->](-0.47,-0.9)node[above] {$\mathbf{b_2}$} (0,0)--(-0.5,-1.5);
%	\draw [ultra thick,->](0.5,-1)node[above] {$\mathbf{b_4}$} (0,0)--(0.25,-1);
%	\draw [thick] (0,0)--(-1.17,-3.5);
%	\draw [thick,dashed,<-] (-2.5,-0.5)node[above] {$x_1$}--(0,0);
%	\draw [thick,dashed] (0,0)--(5,1);
%	\draw [thick,dashed] (0,0)--(-2,2.5);
%	\draw [thick,dashed,->] (0,0)--(3,-3.75)node[right] {$x_2$};
%	\draw [line width=2pt] (0,5)--(0,6);
%	\draw [line width=2pt] (3,3.5)--(3,6);
%	\draw [line width=2pt] (-1,-2)--(-1,4);
%	\draw [line width=2pt] (-1,-2)--(-1.5,-3.5);
%	\draw [line width=2pt] (3,3.5)--(6,3);
%\end{tikzpicture}

Let $X(\Ssig)^I$ be a toric variety with $I=\{i_1,\dots,i_m\}$ and $m\leq r$. Given $u=(u_1,\dots,u_r)\in X(\Ssig)^*{}^I$, then $u_{i_1},\dots,u_{i_m}\in\C^*$ and $u_{i_j}=0$ for $i_j\not\in I$. Therefore, we define
$$
\C^{*\{1,\dots,r\}}_u:=\{(z_1,\dots,z_r)\in(\C^*)^r:z_{i_j}=u_{i_j} \;\mbox{for}\;1\leq j\leq m\}.
$$

%Let $X(\Ssig)^I$ be an orbit with $I=\{i_1,\dots,i_m\}$ and $m\leq r$. Observe that given $U=(U_1,\dots,U_r)\in X(\Ssig)^*{}^I$, there exist $u_{i_1},\dots,u_{i_m}\in\C^*$ such that $U_{i_j}=u_{i_j}$ for $i_j\in I$ and $U_{i_j}=0$ for $i_j\not\in I$. Therefore, we define
%$$
%\C^{*\{1,\dots,r\}}_U:=\{(z_1,\dots,z_r)\in(\C^*)^r:z_{i_j}=u_{i_j} \;\mbox{for}\;1\leq j\leq m\}.
%$$
\begin{defn}
    Let $\Delta\subset\Gamma_{nc}(g;z)$ be an essential non-compact face of the polynomial function $g$ on $X(\Ssig)$, and let $w=(w_1,\dots,w_n)\in\sigma$ be a weight vector satisfying the conditions (i) and (ii) of the Definition \ref{essenc}. Suppose that 
    $I_\Delta=I_w=\{i_1,\dots,i_m\}$, i.e. $\lip w,b_i\rip=0$, if and only if $i=i_1,\dots,i_m$, and $I_w\in\mathscr{I}_v(g)$. We say that the face function
    $$
    g_\Delta(z)=\sum_{\lambda\in\Delta\cap S_{\sigma}}a_\Lambda z^{\Lambda}
    $$
    is \emph{locally tame} if there exists a positive number $r(g_\Delta)>0$ such that for any $u\in {X(\Ssig)^*}^I$ with
    \begin{equation}\label{desg}
    |u_{i_1}|^2+\cdots+|u_{i_m}|^2<r(g_{\Delta})^2,    
    \end{equation}
    $g_\Delta$ has no critical points in $\C^{*\{1,\dots,r\}}_{u}$ as a function of the ($r-m$)-variables $z_{i_{m+1}},\dots,z_{i_r}$, where $\{i_{m+1},\dots,i_r\}:=\{1,\dots,r\}\setminus\{i_1,\dots,i_m\}$ and $u_{i_1},\dots,u_{i_m}\in\C^*$ are the coordinates of $u$. We say that $g$ is \emph{locally tame along a vanishing variety} $X(\Ssig)^I$ if, for any essential non-compact face $\Delta$ with $I_\Delta=I$, the function $g_\Delta$ is locally tame. Finally, we say that $g$ is \emph{locally tame along the vanishing  varieties} if it is locally tame along $X(\Ssig)^I$ for any $I\in\mathscr{I}_v$.
\end{defn}

\begin{obs}
    Given $I\in\mathscr{I}_v(g)$, we use the following notation: let $\overline{r}(g)$ denote the supremum of the numbers $r(g_\Delta)$ satisfying condition \ref{desg}. We define
    $$r_I(g):=\dsp\inf_{I_\Delta=I} \overline{r}(g_\Delta) \ \ \text{and} \ \ r_{nc}(g):=\dsp\inf_{I\in\mathscr{I}_v} r_I(g).$$
\end{obs}

%\begin{exem}
%    Consider $\sigma=\lip e_2,2e_1-e_2\rip$. Then $\sigma^\vee\lip e_1 ,e_1 +2e_2 \rip$, $S_\sigma=\lip e_1 ,e_1 +2e_2 ,e_1 +e_2 \rip$ and the addicting relations between elements of $S\sigma$ are $b_1+b_2=2b_3$. Therefore $X(\Ssig)$ is a $2$-variety in $\C^3$ generated by binomial $z_1z_2-z_3^2$ and your canonical embedding is $h(\xi_1,\xi_2)=(\xi_1,\xi_1\xi_2^2,\xi_1\xi_2)$. Consider the polynomial $g(z)=z_1z_2^2z_3+z_1^2z_2z_3+3z_1^3z_2z_3$. Using the canonical embedding it follow that $L(g)(\xi)=\xi_1^4\xi_2^5+\xi_1^4\xi_2^3+3\xi_1^5\xi_2^3$ and your support will be $supp(L(g))=\{(4,5),(4,3),(5,3)\}$. Then $\Gamma_{nc}(g)$ has two essential non-compact face $\Delta_1=A+\R_+\cdot b_2$ and $\Delta_2=B+\R_+\cdot b_1$, where $A=(4,5)$ and $B=(4,3)$, here $I_{\Delta_1}=\{2\}$ and $I_{\Delta_2}=\{1\}$. Note that for any $U_2\in\C^*$ the function $g_{\Delta_1}(z_1,U_2,z_3)=z_1U_2^2z_3$ of the variable $z_1$ and $z_3$ has no critical points in $\C^{*\{1,2,3\}}_{U_2}$. Thus $g_{\Delta_1}$ is locally tame along of vanish orbit $X(\Ssig)^{\{2\}}$ ($r(g_{\Delta_1})=\infty$). Now for any $U_1\in\C^*$, such that $|u_1|<\dfrac{1}{3}$, the function $g_{\Delta_2}(U_1,z_2,z_3)=U_1^2z_2z_3+3u_1^3z_2z_3$ of the variable $z_2$ and $z_3$ has no critical points in $\C^{*\{1,2,3\}}_{u_1}$, hence $g_{\Delta_2}$ is locally tame along of vanish orbit $X(\Ssig)^{\{1\}}$ ($r(L(g_{\Delta_2}))=\dfrac{1}{3}$). Therefore, $g$ is locally tame along the vanish orbit.
%\end{exem}

\begin{exem} Consider $X(\Ssig)$ as in Example \ref{example1} and let $g(z_1,z_2,z_3)=z_1^2z_3^2-z_2^3z_3^2+z_3^3$ be a polynomial function on $X(\Ssig)$. There is only one essential non-compact face: $\Delta=A+\R_+b_1$, where $A=(4,2)$. We have $I_{\Delta}=\{1\}$, thus given $u=(u_1,0,0)\in X(\Ssig)^*{}^{\{1\}}$, the function $g_{\Delta}(u_1,z_2,z_3)=u_1^2z_3^2$, has no critical points in $\C_{u}^{*\{1,2,3\}}$. Hence, $g$ is locally tame along the vanishing  variety $X(\Ssig)^*{}^{\{1\}}$. 
%Analogously, since $I_{\Delta_2}=\{2\}$, $g$ is locally tame along the vanishing  variety $X(\Ssig)^*{}^{\{2\}}$. 
 However, considering the cone $\sigma=\lip e_1,e_2,e_3\rip$, then $S_{\sigma}=\lip e_1,e_2,e_3\rip \subset \mathbb{Z}^3$, and $X(S_{\sigma})=\C^3$. The same polynomial function $G(z_1,z_2,z_3)=z_1^2z_3^2-z_2^3z_3^2+z_3^3$ on $\mathbb{C}^3$ is not locally tame along the vanishing  varieties (see \cite[Example $2.8$]{eyraloka1}). 
\end{exem}

\section{Local tameness and admissible families}

%Throughout the text we will use the following notation. 

Let $(t,z):=(t,z_1,\dots, z_r)$ be coordinates for $ \C\times X(\Ssig)$, let $0\in U\subset X(\Ssig)$ be an open set containing $0$, let $0\in D\subset\C$ be an open disc, and let
$$
\begin{array}{cccc}
    f: & (D\times U, D\times\{0\}) & \longrightarrow & (\C,0)  \\
       &       (t,z)               &    \longmapsto  & f(t,z)
\end{array}
$$
be a polynomial function. With this notation, $f(D\times\{0\})=0$. We write $f_t(z)=f(t,z)$ and $V(f_t)\subset U$ for the hypersurface defined by $f_t(z)$. As in the previous section, $f$ is the restriction of a polynomial function $F:\C\times\C^r\longrightarrow\C$ to the toric variety $\C\times X(\Ssig)$.

In \cite[Proposition 3.1]{eyraloka1},  Eyral and Oka proved an uniform version
of \cite[ChapterIII, Lemma(2.2)]{okabook} and \cite[Theorem 19]{Oka3}. Here we
extend \cite[Proposition 3.1]{eyraloka1} to families of non-degenerate hypersurfaces on toric varieties.

\begin{prop}\label{propsmooth}
    Suppose that for all $t$ sufficiently small, the following conditions are satisfied:
    \begin{enumerate}
        \item [(i)] the compact Newton boundary $\Gamma(f_t;z)$ is independent of $t$,
        \item [(ii)] $f_t$ is non-degenerate.
    \end{enumerate}
    Then there exists a positive number $E>0$ such that for any $I\in\mathscr{I}_{nv}(f_0)$ and any $t$ sufficiently small, the set $V(f_t)\cap {X(\Ssig)^*}^I\cap B_E$ is non-singular and intersects transversely with $S_\varepsilon$ for any $\varepsilon<E$, where $B_E$ and $S_\varepsilon$ are the open ball and sphere centered at the origin $0\in\C^r$ and radius $E$ and $\varepsilon$, respectively. 
\end{prop}
\begin{proof}
    Since the number of subsets $I\in \mathscr{I}_{nv}(f_0)$ is finite, it is sufficient to consider a fixed $I\in \mathscr{I}_{nv}(f_0)$ and prove the result for this fixed $I$. For simplicity, we may suppose $I=\{1,\dots,m\}$.

    First, we show the smoothness. Suppose that there exists a sequence of points $\{(t_N,z_N)\}\subset V(f_t)\cap\big(D\times {X(\Ssig)^*}^I\big)$ with
    $$
    (t_N,z_N)\rightarrow (0,0) 
    $$
    and for each $N\in \N$, $z_N$ is a critical point of the restriction $f_{t_N}\rvert_{X(\Ssig)^I}$. Then $(0,0)$ is in the closure of the set
    $$
    W=\{(t,z)\in D\times {X(\Ssig)^*}^I: {f_t}\rvert_{X(\Ssig)^I}(z)=0\;\mbox{and}\;d{f_t}\rvert_{X(\Ssig)^I}\equiv0\}.
    $$
    Consider $h:(\C^*)^n\longrightarrow X(\Ssig)$ the canonical homomorphism as in (\ref{canonicalh}) and $h^I:(\C^*)^n\longrightarrow X(\Ssig)^I$ its restriction to $X(\Ssig)^I$ as in (\ref{canonicalhI}). Then we can define the set 
    $$
    \mathcal{W}:=\{(t,\xi)\in D\times(\C^*)^n:  h^I(\xi)\in {X(\Ssig)^*}^I,\; L^I_{f_t}(\xi)=0\;\mbox{and}\;dL^I_{f_t}\equiv0\}.
    $$
    Since $\check{\sigma} \subset \mathbb{R}^n_+$, the coordinate functions of $h^I$ are monomial functions.  
    Moreover, $h^I$ is surjective when restricted to ${X(\Ssig)^*}^I$ and $(0,0)$ lies in the closure of $W$. Therefore, $(0,0)\in\C\times \C^n$ belongs to the closure of $\mathcal{W}$.   %\input{justificativa0acumula}
    By the Curve Selection Lemma \cite{milnorcurvelemma}, there is a real analytic curve $p:[0,\varepsilon)\longrightarrow D\times \C^n$ such that 
    \begin{eqnarray}
        p(0) & = & (t(0),\xi(0))=(0,0) \label{z=0}\nonumber \\ 
        p(s) & = & (t(s),\xi(s))=(t(s),\xi_1(s),\dots,\xi_n(s))\in \mathcal{W}\;\mbox{if}\;s\not=0.\nonumber
    \end{eqnarray}
    For each $1\leq i\leq n$, consider the Taylor expansion 
    \begin{align*}
    t(s)    &=t_0s^{v_0}+\cdots \\
    \xi_i(s)&=a_is^{w_i}+\cdots    
    \end{align*}
    where $v_0,w_i\in\Z_+^*$, $t_0,a_i\in\C^*$ for $1\leq i\leq n$, and ``dots'' represent the higher-order terms. Let $a=(a_1,\dots,a_n)\in(\C^*)^n$ and $w=(w_1,\dots,w_n)$. Let $\Delta\subset\Gamma(L^I_{f_t};\xi)$ be a compact face of $L^I_{f_t}$ with respect to $\xi$ in $\R_+^n$, defined by the set 
    %$$\Delta = \{x\in \Gamma_+(L^I_{f_t};\xi); \ \ \langle x,w\rangle=\displaystyle\sum_{i=1}^{n}x_iw_i=d\}$$
    %with $d$ the minimal value of $\langle x,w\rangle$
    satisfying $\langle x,w\rangle=\displaystyle\sum_{i=1}^{n}x_iw_i=d$ with $x\in \Gamma_+(L^I_{f_t};\xi)$ and $d$ its minimal value as in \cite{okabook}. Since compact faces of $L^I_{f_t}$ are compact faces of $f_t{}\rvert_{X(\Ssig)^I}$ (see Remark \ref{obsfc}), we can consider $\Delta$ as a compact face in $\Gamma(f_t{}\rvert_{X(\Ssig)^I};z)
    =\Gamma(f_0{}\rvert_{X(\Ssig)^I};z)$. For each $1\leq i\leq n$,

    %Since $L^I(f_t)$ and $f_t{}\rvert_{X(\Ssig)^I}$ share (compacts) faces, we can consider $\Delta$ as face in $\Gamma(f_t{}\rvert_{X(\Ssig)^I};z)=\Gamma(f_0{}\rvert_{X(\Ssig)^I};z)$. For any $1\leq i\leq n$,
    $$
    \dfrac{\partial(L^I_{f_{t(s)}})}{\partial \xi_i}(\xi(s))=\dfrac{\partial((L^I_{f_{t(s)}})_\Delta)}{\partial\xi_i}(a)s^{d-w_i}+\cdots
    $$
    where $(L^I_{f_t})_\Delta$ is the polynomial function associated with the face function $(f_t)_\Delta{}\rvert_{X(\Ssig)^I}$, and $\xi(s)\in\mathcal{W}$ for $s\neq0$. Thus, 
    $$
    0=\dfrac{\partial((L^I_{f_{t(s)}})_\Delta)}{\partial\xi_i}(a)s^{d-w_i}+\cdots \;\;\mbox{for}\;\;s\neq0,
    $$
    which implies that for every $1\leq i\leq n$
    $$
    \dfrac{\partial((L^I_{f_{t(s)}})_\Delta)}{\partial\xi_i}(a)=0.
    $$
    Consequently 
    $$
    \dfrac{\partial((L^I_{f_0})_\Delta)}{\partial\xi_i}(a)=0.
    $$
    Therefore, $a\in(\C^*)^n$ is a critical point of $(L^I_{f_{0}})_\Delta$. This implies that ${f_0}\rvert_{X(\Ssig)^I}$ is not non-degenerate, which contradicts Proposition \ref{propdeg}, since the polynomial function $f_0$ is non-degenerate and must remain non-degenerate on every $X(\Ssig)^I$ for $I \in \mathscr{I}_{nv}(f_0)$. 
    
    We also use a contradiction argument to prove transversality. Assume the existence of a sequence of points $\{(t_N,\xi_N)\}\subset D\times \C^n$, with $(t_N,\xi_N)\rightarrow(0,0)$ such that for all $N$, $$(t_N,h^I(\xi_N))\in V(f)\cap(D\times X(\Ssig)^*{}^I)\subset D\times\C^r$$ and does not intersect the sphere $S_{\| h^I(\xi_N)\|}$ transversely in $h^I(\xi_N)$. Therefore, $(0,0)\in D\times\C^n$ belongs to the closure of the set
    $$
    \widetilde{\mathcal{W}}=\{(t,\xi)\in D\times (\C^*)^n: f_t\rvert_{X(\Ssig)^I}(h^I(\xi))=0\;\mbox{and}\;\nabla f_t\rvert_{X(\Ssig)^I}(h^I(\xi))=\lambda h^I(\xi),\;\lambda\in\C^*\}, 
    $$
    in which
    $$
    \nabla {f_t\rvert_{X(\Ssig)^I}}(h^I(\xi))=\left(\; \overline{\dfrac{\partial {f_t\rvert_{X(\Ssig)^I}}}{\partial z_1}}(h^I(\xi)),\dots,\overline{\dfrac{\partial {f_t\rvert_{X(\Ssig)^I}}}{\partial z_m}}(h^I(\xi)),0,\dots,0 \;\right)
    $$ 
    is the gradient vector of $f_t\rvert_{X(\Ssig)^I}$ and the bar denotes the complex conjugate. By the Curve Selection Lemma \cite{milnorcurvelemma}, there is a real analytic curve $(t(s),\xi(s))=(t(s),\xi_1(s),\dots,\xi_n(s))$, with $s\in[0,\varepsilon)$, such that
    \begin{enumerate}
        \item $(t(0),\xi(0))=(0,0)$;
        \item  $(t(s),\xi(s))\in D\times(\C^*)^n$ for $s\neq0$;
        \item  $f_{t(s)}(h^I(\xi(s))=0$;
        \item  \label{it:item4} $\nabla f_{t(s)}\rvert_{X(\Ssig)^I}(h^I(\xi(s)))=\lambda(s) h^I(\xi(s))$.
    \end{enumerate}
        Consider the Taylor expansions $t(s)=t_0s^{v_0}+\cdots$ and $\xi_i(s)=a_is^{w_i}+\cdots$ for $1\leq i\leq n$, where $v_0,w_i\in\Z_+^*$, $t_0,a_i\in\C^*$ for $1\leq i\leq n$, and the Laurent expansion $\lambda(s)=\lambda_0s^u+\cdots$, in which $\lambda_0\in\C^*$ and $u \in\Z_+^*$. Note that $$h^I(\xi(s))=(z_1(s),\dots,z_m(s),0,\dots,0)\in\C^r$$ is a curve with $z_i(s)=\xi(s)^{b_i}$ for $1\leq i\leq n$. Thus, writing $a=(a_1,\dots,a_n)\in(\C^*)^n$ and $w=(w_1,\dots,w_n)$, we have $z_i(s)=a^{b_i}s^{W_i}+\dots$, where $W_i=\lip w,b_i\rip$ for $1\leq i\leq m$.
        
        Let $\Delta\subset\Gamma(f_t\rvert_{X(\Ssig)^I};z)=\Gamma(f_0\rvert_{X(\Ssig)^I};z)$ be the compact face associated with $w$ and $$d=\dsp\min_{x\in\Gamma(f_t\rvert_{X(\Ssig)^I};z)}\lip w,x\rip.$$ After reordering, we may assume without loss of generality that \( W_1 = \cdots = W_k < W_j \), where \( k < j \leq m \).
 Now, we observe that Lemma \ref{pesos} holds true for the vector $W=(W_1,\dots,W_m,0,\dots,0)$. Thus, by item \ref{it:item4}. above, we have $d-W_1=u+W_1$ and 
        %Note that,  is a vector which holds the Lemma \ref{pesos}
        \begin{equation}\label{eqderi}
            \overline{\dfrac{\partial {((f_0)_\Delta\rvert_{X(\Ssig)^I})}}{\partial z_i}}(h^I(a))=
        \begin{cases}
            \lambda_0a^{b_i}, &\;\mbox{for}\;1\leq i\leq k\\
            0,                &\;\mbox{for}\;k<i\leq m
        \end{cases}
        .
        \end{equation}
        
    As the polynomial function $(f_0)_\Delta\rvert_{X(\Ssig)^I}$ is weighted homogeneous with respect to the weights \(W\) and degree \(d\), by Euler’s identity, it follows that \begin{equation}\label{ideuler}
        d\cdot((f_0)_\Delta\rvert_{X(\Ssig)^I})(h^I(a))=\sum_{i=1}^{m}W_ia^{b_i}\dfrac{\partial {((f_0)_\Delta\rvert_{X(\Ssig)^I})}}{\partial z_i}(h^I(a)).
    \end{equation}
    Since $f_{t(s)}(z(s))=0$ for any $s\in[0,\varepsilon)$, we have $(f_0)_\Delta\rvert_{X(\Ssig)^I}(h^I(a))=0$. Then, by (\ref{ideuler})
    \[
    0=\sum_{i=1}^{m}W_ia^{b_i}\dfrac{\partial {((f_0)_\Delta\rvert_{X(\Ssig)^I})}}{\partial z_i}(h^I(a)),
    \]
    and applying (\ref{eqderi}), we get the contradiction
    \[
    0=\overline{\lambda_0}\sum_{i=1}^{k}W_i\|a^{b_i}\|^2,
    \]
    since, $W_i>0$ for $1\leq i\leq k$. 
\end{proof}

\begin{obs}\label{helpteo1}
    Following the notation of Proposition~\ref{propsmooth}, since the number of subsets \( I \in \mathscr{I}_{nv}(f_0) \) is finite, we can take \( \overline{E} \) to be the infimum of all values \( E \) for which
\[
V(f_t) \cap {X(\Ssig)^*}^I \cap B_E
\]
is non-singular and intersects transversely with \( S_\varepsilon \) for any \( \varepsilon < E \). Therefore,
\[
\bigcup_{I \in \mathscr{I}_{nv}(f_t)} V(f_t) \cap {X(\Ssig)^*}^I \cap B_{\overline{E}}
\]
is non-singular and intersects transversely with \( S_\varepsilon \) for any \( \varepsilon < \overline{E} \). 
\end{obs}

%Let be $f:(D\times U,(0,0)\longrightarrow(\C,0)$ polynomial germ as in previous section. 
From now on, unless stated otherwise, we assume that for all $t$ sufficiently small, the following conditions hold:
\begin{enumerate}
    \item[(I)] $\Gamma_{nc}(f_t:z)$ (consequently, $\mathscr{I}_{nv}(f_t)$ and $\mathscr{I}_v(f_t)$) is independent of $t$;
    \item[(II)] $f_t$ is non-degenerate and locally tame along the vanishing varieties.
\end{enumerate}
By Proposition \ref{propsmooth} there exists $E>0$ such that for any $I\in \mathscr{I}_{nv}(f_t)$ (which is equal to $\mathscr{I}_{nv}(f_0))$ and $t$ sufficiently small
$$
V(f_t)\cap X(\Ssig)^*{}^I\cap B_E
$$
is non-singular. Therefore, in a sufficiently small open neighborhood $\mathscr{U}\subset D\times U$ of the origin in $\C\times X(\Ssig)$, the set $V(f)\cap (\C\times X(\Ssig)^*{}^I)$ is non-singular for any $I\in \mathscr{I}_{nv}(f_t)$. In such neighborhood $\mathscr{U}$ we can stratify $\C\times X(\Ssig)$ so that the hypersurface $V(f_t)$ is the union of strata. Namely, we will consider three types of strata:
\begin{enumerate}
    \item [(a)] $A_I=\mathscr{U}\cap[V(f)\cap(\C\times X(\Ssig)^*{}^I)]$, for $I\in\mathscr{I}_{nv}(f_0)$;
    \item [(b)] $B_I=\mathscr{U}\cap[(\C\times X(\Ssig)^*{}^I)\setminus V(f)\cap(\C\times X(\Ssig)^*{}^I)]$, for $I\in\mathscr{I}_{nv}(f_0)$;
    \item [(c)] $C_I=\mathscr{U}\cap(\C\times X(\Ssig)^*{}^I)$, for $I\in\mathscr{I}_v(f_0)$.
\end{enumerate}
The finite collection
$$
\mathscr{T}:=\{A_I,B_I:I\in\mathscr{I}_{nv}(f_0)\}\cup\{C_I: I\in\mathscr{I}_{v}(f_0)\}
$$
defines a stratification of $\mathscr{U}\cap\C\times X(\Ssig)^*{}^I$ such that $\mathscr{U}\cap V(f)$ is a union of strata. Note that for $I=\emptyset\in\mathscr{I}_v(f_0)$ the stratum $C_{\emptyset}=\mathscr{U}\cap(\C\times \{0\})$ is the $t$-axis. The collection $\mathscr{T}$ is called \emph{canonical stratification}.
\begin{defn}
    The family $\{f_t\}$ is said to be \emph{admissible} at $t=0$ if it satisfies the conditions (I) and (II) above and there exists a positive number $\rho>0$ such that for any sufficiently small $t$, $\inf\{E,r_{nc}(f_t)\}\geq\rho$, where $E$ is given by Proposition \ref{propsmooth}.
\end{defn}
If the family ${f_t}$ is admissible, then it is \emph{uniformly locally tame} along the vanishing varieties, as discussed in \cite[pg $104$]{eyraloka1}.

\section{Admissibility and Equisingularity}
This section is devoted to the proof of the main result, presented in the following theorem.
\begin{teo}\label{mainteo}
    If the family of polynomial functions $\{f_t\}$ (as in the previous section) is admissible, then the canonical stratification $\mathscr{T}$ of $\mathscr{U}\cap (\C\times X(\Ssig))$ is a Whitney stratification. Therefore, the corresponding family of hypersurfaces $\{V(f_t)\}$ is Whitney equisingular.
\end{teo}
Prior to the proof of the main theorem, we briefly recall the notion of Whitney equisingularity, along with some relevant observations. A stratification \( \mathcal{S} \) of a subset of $\C^r$ is said to be a \emph{Whitney stratification} if, for each stratum $S$, both its closure $\overline{S}$ and $\overline{S}\setminus S$ are analytic sets. Moreover, for any pair of strata $(S_1,S_2)$ and any $p\in S_1\cap\overline{S}_2$, $S_2$ is Whitney $(b)$-regular over $S_1$ at the point $p$. That is, for any sequence of points $\{p_k\}\subset S_1$, $\{q_k\}\subset S_2$ and $\{a_k\}\subset \C$ satisfying:
\begin{enumerate}
    \item [(i)] $p_k\to p$ and $q_k\to p$;
    \item [(ii)] $T_{q_k}S_2\to T$;
    \item[(iii)] $a_k(p_k-q_k)\to v$;
\end{enumerate}
we have $v\in T$. Here $T_{q_k}S_2$ is the tangent space of $S_2$ at the point $q_k$ and the convergence in (ii) occurs in the respective Grassmannian. Observe that we do not assume the frontier condition to be satisfied. Nevertheless, if \( \mathcal{S} \) is a Whitney stratification of \( \mathcal{U} \cap V(f) \) such that the \( t \)-axis is one of its strata, then the partition \( \mathcal{S}^c \), formed by the connected components of the strata of \( \mathcal{S} \), is likewise a Whitney stratification; furthermore, \( \mathcal{S}^c \) does satisfy the frontier condition (we refer to \cite{gibsontopological} for further details).

\begin{obs}\label{helpteo2}
    If $M$ is a smooth manifold and $N\subset M$ is a closed smooth submanifold of $M$, then $M\setminus N$ is Whitney $(b)$-regular over $N$ at any point.
\end{obs}
%\textcolor{blue}{O proximo exemplo é não normal!!!!!!!}
%\begin{obs}\label{goodexem}
 %   As in \cite{eyraloka1}, our definition of admissibility depends on the notion of non-degeneracy. However, in the toric setting, non-degeneracy  can include a broader class of polynomial functions. For instance, in $\C^3$ the polynomial function $f(z_1,z_2,z_3)=(z_2^3-z_1^2)z_3$ is not non-degenerate, since the dimension of $V(z_2^3-z_1^2)\cap V(z_3)$ is greater than $1$ and its Newton polyhedron intersects the coordinate axes (see \cite{kouchnirenko1976}). On the other hand, if we consider the toric variety \( X(\Ssig) \subset \mathbb{C}^3 \) associated with the semigroup \( \Ssig = \langle e_1, e_1 + e_2, 2e_2 \rangle \), the same polynomial becomes non-degenerate on \( X(\Ssig) \). Indeed, in this case $f$ has a unique compact face, namely $\Delta= A$ where $A=(2,2)$. Consequently, $L(f_\Delta)(\xi_1,\xi_2)=-\xi_1^2\xi^2$ has no critical points on $(\C^*)^2$. 
%\end{obs}

\begin{proof}[Proof of Theorem \ref{mainteo}]
    First, note that if $I\subset J$, then $X(\Ssig)^*{}^I\subset \overline{X(\Ssig)^*{}^J}$. Moreover, observe that if $I\subset J$ and $J\in\mathscr{I}_v(f_0)$, then $I\in\mathscr{I}_v(f_0)$ as well. Thus, we need to check the Whitney $(b)$-regularity only for the pairs of strata which satisfy the following conditions:
    \begin{itemize}
        \item $C_I\cap\overline{C_J}\neq\emptyset$, with $I\subset J$ and $I,J\in\mathscr{I}_v(f_0)$;
        \item $C_I\cap\overline{A_J}\neq\emptyset$ or $C_I\cap\overline{B_J}\neq\emptyset$, with $I\subset J$ and $I\in\mathscr{I}_v(f_0)$, $J\in\mathscr{I}_{nv}(f_0)$;
        \item $A_I\cap\overline{A_J}\neq\emptyset$, $A_I\cap\overline{B_J}\neq\emptyset$ or $B_I\cap\overline{B_J}\neq\emptyset$, with $I\subset J$ and $I,J\in\mathscr{I}_{nv}(f_0)$.
    \end{itemize}
    To prove our result, it suffices to show that for any $J\in\mathscr{I}_{nv}(f_0)$ and $I\in\mathscr{I}_{v}(f_0)$, with $I\subset J$,  the stratum $A_J=\mathscr{U}\cap(V(f)\cap X(\Ssig)^*{}^J)$ is Whitney (b)-regular over $C_I=\mathscr{U}\cap(\C\times X(\Ssig)^*{}^I)$.  
    For the other cases, the Whitney $(b)$-regularity follows from Remark \ref{helpteo2}. For instance, consider the pair of strata $(A_I,A_J)$. According to Remark \ref{helpteo1} the set $A^\# =\cup_{K\in\mathscr{I}_{nv}(f_0)} A_K$ is non-singular. Furthermore, $\overline{A_J}\cap A^\#$ contains a smooth closed submanifold $S\subset A^\#$ which includes $A_I$ and $S\cap A_J=\emptyset$. Applying Remark \ref{helpteo2}, with $M=A^\#$ and $N=S$, ensures the Whitney $(b)$-regularity.
    
    Without loss of generality, we can assume $I=\{1,\dots,m\}$ and $J=\{1,\dots,r\}$ with $1\leq m<r$. 
    It suffices to verify that the Whitney (b)-regularity condition holds along arbitrary real analytic curves
    \begin{align*}
        \gamma(s)         &=      (t(s),z(s))            = (t(s),z_1(s),\dots,z_r(s))\\
        \tilde{\gamma}(s) &= (\tilde{t}(s),\tilde{z}(s)) = (\tilde{t}(s),\tilde{z}_1(s),\dots,\tilde{z}_r(s))
    \end{align*}
where $\gamma(s)\in A_J$ and $\Tilde{\gamma}(s)\in C_I$ for $s\neq0$, and $\gamma(0)=\tilde{\gamma}(0)=(\tau,q) \in C_I \cap \overline{A_J}$. 
Since $(\tau, q) \in C_I \cap \overline{A_J}$, it follows that $(\tau,q) = (\tau,q_1,\dots,q_m,0,\dots,0)$, with $q_i\neq0$ for $i=1,\dots,m$.

    Let $l(s):=\vet{\tilde{\gamma}(s)\gamma(s)}$ denote the line segment connecting $\gamma(s)$ to $\tilde{\gamma}(s)$. We aim to show that, as $s\rightarrow0$, the line $l(s)$ belongs to the tangent space $T_{\gamma(s)}A_J$. To this end, we consider the Taylor expansions of the curves $\gamma(s)$ and $\tilde{\gamma}(s)$ around the point $(\tau,q)$. For $\gamma(s)$, we write:
    \begin{align*}
        t(s)           &=  \tau+A_0s+\cdots\\
        z_1(s)         &=  q_1+A_1s+\cdots\\
                   & \;\;\vdots  \\
        z_m(s)         &=  q_m+A_ms+ \cdots\\
        z_{m+1}(s)     &=  A_{m+1}s^{W_{m+1}}+\cdots\\
                   & \;\;\vdots  \\
        z_r(s)         &=  A_rs^{W_r}+\dots
    \end{align*}
where $W_i\in\Z_+^*$, $A_i\in\C^*$, for $m+1\leq i\leq r$. And for $\tilde{\gamma}(s)$, we write:
    \begin{align*}
        \tilde{t}(s)   &= \tau+\tilde{A_0}s+\cdots\\
        \tilde{z}_i(s) &=q_i+\tilde{A_i}s+\cdots
    \end{align*}
    where $i\in\{1,\dots,r\}$, and $\tilde{z}_i(s)=0$ (for all $s$) for $m+1\leq i\leq r$. 
    %Furthermore, for any $s$,  for $m+1\leq i\leq r$. 
    Thus, writing $l(s)=(l_0(s),l_1(s),\dots,l_r(s))$, we have
    $$
    l_i(s)=
    \begin{cases}
        (A_0-\tilde{A_0})s+\cdots, & \mbox{for}\;i=0\\
        (A_i-\tilde{A}_i)s+\cdots, & \mbox{for}\;1\leq i\leq m\\
        A_is^{W_i}+\cdots,         & \mbox{for}\;m+1\leq i\leq r 
    \end{cases}
    .
    $$
    Without loss of generality, we may suppose    
    \begin{equation}\label{deswi}
    \begin{split}
    W_{m+1}=\cdots=W_{m+m_1}<W_{m+m_1+1}=\cdots=W_{m+m_1+m_2}<\cdots\\
    \cdots<W_{m+m_1+\cdots+m_{k-1}+1}=\cdots=W_{m+m_1+\cdots +m_k}=W_r,\qquad    
    \end{split}
    \end{equation}
    for $m, m_1,\dots,m_k\in\Z_+^*$ and $m + \dsp\sum_{i=1}^{k}m_i=r$. Under these conditions, we will prove the following equality 
    \begin{equation}\label{lim}
        \lim_{s\rightarrow0}\dfrac{\lip l(s),\nabla F(\gamma(s))\rip}{||l(s)||\cdot||\nabla F(\gamma(s))||}=0
    \end{equation}
    where $\lip \cdot, \cdot\rip$ is the Hermitian inner product on $\C\times\C^r$, $f$ is the restriction of a polynomial function $F:\C\times\C^r\longrightarrow\C$ to the toric variety $\C\times X(\Ssig)$
    and 
    $$
    \nabla F(\gamma(s))=\left(\; \overline{\dfrac{\partial F}{\partial t}}(\gamma(s)), \overline{\dfrac{\partial F}{\partial z_1}}(\gamma(s)),\dots,\overline{\dfrac{\partial F}{\partial z_r}}(\gamma(s)) \;\right)
    $$ 
    is the gradient vector of $F$ at $\gamma(s)$ (the bar denotes the complex conjugate). Let us consider
    $$
    \ord l(s)=\dsp\inf_{0\leq i\leq r} \mbox{ord}\;l_i(s)
    $$
    where $\ord l_i(s)$ is the order with respect to $s$ of the $i$-th coordinate of $l(s)$. By (\ref{deswi}), $\ord l(s)\leq W_{m+1}$. Moreover if $\ord l(s)<W_{m+1}$, we have
    \begin{equation}\label{limordmenor}
        \lim_{s\rightarrow0}\dfrac{l(s)}{|s|^{\ord l(s)}}=(\star,\underbrace{\star,\dots,\star}_{m\textrm{-terms}} ,\underbrace{0,\dots,0}_{(r-m)\textrm{-terms}}\!\!\!)        
    \end{equation}
    and if $\ord l(s)=W_{m+1}$, then
    \begin{equation}\label{limordigual}
        \lim_{s\rightarrow0}\dfrac{l(s)}{|s|^{\ord l(s)}}=(\star,\underbrace{\star,\dots,\star}_{m\textrm{-terms}} ,A_{m+1},\dots,A_{m+m_1},\underbrace{0,\dots,0}_{r-(m+m_1)\textrm{-terms}}\!\!\!\!\!\!\!\!)        
    \end{equation}
    where each term ``$\star$'' in (\ref{limordmenor}) and (\ref{limordigual}) represents a complex number, which may be zero or not.

    Let $W:=(0,\dots,0,W_{m+1},\dots,W_r)$, and let $w\in\sigma$ be a weight vector satisfying $\lip w,b_i\rip=0$ for $1\leq i\leq m$. Denote by $\Delta_w$ the face of $\Gamma_{nc}(f_t;z)$ (which is independent of $t$) defined by the locus where the function $\varphi_w:\Gamma_{nc}(f_t;z)\longrightarrow\R$, given by $\varphi_w(x)=\lip x,w\rip$, attains its minimal value. Let $d_w$ denote this minimal value. 
    
    Now observe that, since $X(\Ssig)^I$ is a vanishing  variety, $\Delta_w$ is an essential non-compact face with $I_w=I$. Furthermore, $\Delta_w$ is contained in a hyperplane parallel to the face of the dual cone $\sigduo$ generated by $b_1,\dots,b_m$. Consequently, if $\lambda\in\Delta_w\cap\Ssig$, then 
    $$
    \lambda=p+\sum_{i=1}^{m}X_ib_i,
    $$
    where $p$ is a fixed point in $\Delta_w\cap\Ssig$ and $X_i\in\Z_+$ for $1\leq i\leq m$, or alternatively 
    $$
    \lambda=\sum_{i=i}^{r}\Lambda_ib_i+\sum_{i=1}^{m}X_ib_i,
    $$
    in which $\Lambda_i\in\Z_+$ for $1\leq i\leq r$, and at least one $\Lambda_i\neq0$ for some $m+1\leq i\leq r$. For this reason, the representation of $\lambda$ in $\N^r$ (see page $5$) has the following form
    $$
    \Lambda:=(\Lambda_1+X_1,\dots,\Lambda_m+X_m,\Lambda_{m+1},\dots,\Lambda_r).
    $$
    As a consequence, the inner product $\lip W,\Lambda\rip$ does not depend on the $X_i$'s and is therefore constant. We denote it by $D_W$.
     Let $A=(q_1,\dots,q_m,A_{m+1},\dots,A_r)$. Since $\Gamma_{nc}(f_t;z)$ is independent of $t$,
    \begin{equation}\label{gradinz}
    \dfrac{\partial F}{\partial z_i}(\gamma(s))=\dfrac{\partial (F_\tau)_{\Delta_w}}{\partial z_i}(A)\ s^{D_W-W_i}+\cdots 
    \end{equation}
    for each $1\leq i\leq r$, while
    \begin{equation}\label{limt}
        \lim_{s\rightarrow0}\left(\dfrac{1}{|s|^{D_W-1}}\ \dfrac{\partial F}{\partial t}(\gamma(s))\right)=0,
    \end{equation}
    where $(F_\tau)_{\Delta_w}$ represents the function face $(f_\tau)_{\Delta_w}$ (associated with $f_\tau$ and $\Delta_w$) viewed as a polynomial function on $\C^r$. Now, by (\ref{deswi})
    \begin{equation}\label{desdw}
        \begin{split}
            D_W-W_{m+1}=\cdots=W_{m+m_1}>D_W-W_{m+m_1+1}=\cdots=D_W-W_{m+m_1+m_2}>\cdots\\
    \cdots >D_W-W_{m+m_1+\cdots+m_{k-1}+1}=\cdots=D_W-W_{m+m_1+\cdots+ m_k}=D_W-W_r.\qquad
        \end{split}
    \end{equation}
    Let us denote by 
    $$
    o(s):=\ord \overline{\nabla F}(\gamma(s))=\inf\left\{\ord \dfrac{\partial F}{\partial t}(\gamma(s)), \inf_{1\leq i\leq r}\ord \dfrac{\partial F}{\partial z_i}(\gamma(s))\right\}.
    $$
Assuming $(\tau,q)$ is sufficiently close to $(0,0)\in\C\times X(\Ssig)$, uniform local tameness ensures the existence of $i_0\in\{m+1,\dots,r\}$ such that \begin{equation}\label{diffneq0}
        \dfrac{\partial (F_\tau)_{\Delta_w}}{\partial z_{i_0}}(A)\neq0.
    \end{equation}
Thus, combining equations (\ref{gradinz}) and (\ref{desdw}), the relation (\ref{diffneq0}) implies that
    $$
    o(s)\leq D_W-W_{i_0}\leq D_W-W_{m+1}\leq D_W-1.
    $$
    Consequently, from (\ref{limt}) we have
    $$
    \lim_{s\rightarrow0}\left(\dfrac{1}{|s|^{o(s)}}\ \dfrac{\partial F}{\partial t}(\gamma(s))\right)=0.
    $$
    Furthermore, since $W_i=0$ for $1\leq i\leq m$, if $o(s)<D_W-W_{m+1}$ then
    \begin{equation}\label{limgrad<}
        \lim_{s\rightarrow0}\dfrac{\nabla F(\gamma(s))}{|s|^{o(s)}}=              (0,\!\!\!\!\underbrace{0,\dots,0}_{(m+m_1)\textrm{-terms}}\!\!\!\!,\!\!\!\!\!\!\overbrace{\star,\dots,\star}^{r-(m+m_1)\textrm{-terms}}\!\!\!\!\!\!),
    \end{equation}
    and if $o(s)=D_W-W_{m+1}$, then we have
    \begin{equation}\label{limgrad=}
            \lim_{s\rightarrow0}\dfrac{\nabla F(\gamma(s))}{|s|^{o(s)}}=
            \Biggl(0,\underbrace{0,\dots,0}_{m\textrm{-terms}},\overline{\dfrac{\partial (F_\tau)_{\Delta_w}}{\partial z_{m+1}}}(A),\dots,\overline{\dfrac{\partial (F_\tau)_{\Delta_w}}{\partial z_{m+m_1}}}(A),\!\!\!\!\!\!\!\!\underbrace{\star,\dots,\star}_{r-(m+m_1)\textrm{-terms}}\!\!\!\!\!\!\!\!\Biggr). 
    \end{equation}
    Since $||l(s)||$ and $||\nabla F(\gamma(s))||$ are equivalent to $c_1|s|^{\ord l(s)}$ and $c_2|s|^{o(s)}$ as $s\rightarrow0$, respectively ($c_1,c_2\in\C^*$), it follows from relations (\ref{limordmenor}), (\ref{limordigual}), (\ref{limgrad<}) and (\ref{limgrad=}) that (\ref{lim}) is immediately satisfied except if $o(s)=D_W-W_{m+1}$ and $\ord l(s)=W_{m+1}$. In this case, we must prove the equality
    \begin{equation}\label{eqosordls}
        \sum_{i=m+1}^{m+m_1}A_i\dfrac{\partial (F_\tau)_{\Delta_w}}{\partial z_i}(A)=0.
    \end{equation}
    For this purpose, note that the polynomial function $(F_\tau)_{\Delta_w}$ is weighted homogeneous with weights $W$ and weight degree $D_W$. Then, by Euler identity
    \begin{equation}\label{eulerid}
        \sum_{i=1}^{r}W_iz_i\dfrac{\partial (F_\tau)_{\Delta_w}}{\partial z_i}(z)=D_W\cdot(F_\tau)_{\Delta_w}(z).
    \end{equation}
    Since $f(\gamma(s))=0$ for any $s$, it follows that $(F_\tau)_{\Delta_w}(A)=0$. Furthermore, as $W_i=0$ for $1\leq i\leq m$ and using (\ref{eulerid}) we have
    \begin{equation}\label{euleridins}
        \sum_{i=m+1}^{r}W_iA_i\dfrac{\partial (F_\tau)_{\Delta_w}}{\partial z_i}(A)=0.
    \end{equation}
    Now the equality $o(s)=D_W-W_{m+1}$, combined with (\ref{gradinz}) and (\ref{desdw}), implies that
    \begin{equation}\label{ftauina}
        \dfrac{\partial (F_\tau)_{\Delta_w}}{\partial z_i}(A)=0
    \end{equation}
    for $m+m_1<i\leq r$. Indeed, if there is $i_1>m+m_1$ such that (\ref{ftauina}) does not hold, then by (\ref{gradinz}) and (\ref{desdw}) we have
    $$
    o(s)\leq D_W-W_{i_1}<D_W-W_{m+1}
    $$
    which is a contradiction. Thus, the next equalities follows from (\ref{euleridins}), (\ref{ftauina}) and (\ref{deswi})
    $$
    \sum_{i=m+1}^{m+m_1}W_iA_i\dfrac{\partial (F_\tau)_{\Delta_w}}{\partial z_i}(A)=0\;\Longrightarrow\; W_{m+1}\sum_{i=m+1}^{m+m_1}A_i\dfrac{\partial (F_\tau)_{\Delta_w}}{\partial z_i}(A)=0.
    $$
    As $W_{m+1}\in\Z_+^*$, the equality (\ref{eqosordls}) holds. This proves (\ref{lim}). Thus, we conclude that as $s \rightarrow 0$, $l(s)$ belongs to the tangent space $T_{\gamma(s)}V(F)$.

    On the other hand, consider the orbit $\Tilde{A_J}:=\C\times X(\Ssig)^*{}^J$. Since $I\subset J$
    \[
    C_I\subset\overline{\Tilde{A_J}}\quad\textrm{and}\quad A_J=V(F)\cap\Tilde{A_J}.
    \]
    Consequently, the curve $\gamma(s)\in\Tilde{A_J}$ for all $s\neq0$. Moreover, as $C_I$ and $\Tilde{A_J}$ are orbits of the torus action, the pair of strata $(C_I, \Tilde{A_J})$ in $\C \times X(\Ssig)$ satisfies the Whitney ($b$)-regularity condition. Hence, as $s \to 0$, the point $l(s)$ belongs to the tangent space $T_{\gamma(s)}\Tilde{A_J}$. Finally, as $T_{\gamma(s)}V(F)\cap T_{\gamma(s)}\Tilde{A_J}\subset T_{\gamma(s)}A_J$ we conclude that 
    \[
    \lim_{s\rightarrow0} l(s)\in T_{\gamma(s)}A_J.
    \]
    %as $s\rightarrow0$, $l(s)$ lies in the tangent space $T_{\gamma(s)}A_J$. 

    Now we will deal with the case $I=\emptyset$. In this case, consider $h^J:(\C^*)^n\rightarrow X(\Ssig)^*{}^J$ given by $h^J(\xi)=(\xi^{b_1},\dots,\xi^{b_m},0,\dots,0)$. Since $h^J$ is surjective, we set
    \[
    \mathscr{A}_J:=\{(\eta,\xi)\in D\times(\C^*)^n:(\eta,h^J(\xi))\in\gamma((0,\varepsilon))\}.
    \]
    Furthermore, since the coordinate functions of $h^J$ are monomials and $\gamma(0)=0$, it follows that $0\in\overline{\mathscr{A}_J}$. Let us define the map $H:D\times(\C^*)^n\times[0,\varepsilon)\rightarrow \C\times\C^r$ given by $H(\eta,\xi,s)=(t(s),h^J(\xi))-\gamma(s)$ and observe that
    \[
    H^{-1}(0)=\{(\eta,\xi,s)\in D\times(\C^*)^n\times[0,\varepsilon): (t(s),h^J(\xi))=\gamma(s)\}.
    \]
    Thus, $H^{-1}(0)$ is an analytic set in $D\times(\C^*)^n\times[0,\varepsilon)$. 
    
    Let $\pi_1:\C\times\C^n\times\R\rightarrow\C\times\C^n$ be the projection given by $\pi_1(\eta,\xi,s)=(\eta,\xi)$, and consider the restriction map $\pi_1\rvert_{H^{-1}(0)}:H^{-1}(0)\rightarrow\C\times(\C^*)^n$. For a given $(\eta,\xi)\in D\times(\C^*)^n$, the preimage $\pi_1^{-1}\rvert_{H^{-1}(0)}(\eta,\xi)$ is a finite set. To ensure this finiteness, we can, if necessary, shrink the interval $[0,\varepsilon)$ appearing in the definition of the map $H$ so that the preimage of any point $z\in A_J$ under the curve $\gamma$ is a finite set.
    %O arqgumento aqui é o seguinte: por $\gamma$ ser uma curva sua pre-imagem é um conjunto dsicreto (possivelmente infinito). Se necessário podemos tomar um $\varepsilon'<\varepsilon$ assim em $[0,\varepsilon']\subset[0,\varepsilon)$ então a preimagem por \gamma sera um conjunto discreto dentro de um compacto, logo é finito. Então consideramos H definida no intervalo $[0,\varepsilon')$.
    %Since $\Gamma$ is analytic curve your image not has double points, therefore $\pi_1^{-1}\rvert_{H^{-1}(0)}$ is finite map. 
    Consequently $\pi_1(H^{-1}(0))=\mathscr{A}_J$ is an analytic set, as stated in Theorem 2 on page 53 of \cite{grauert2012theory}. Then by the Selection Curve Lemma (analytic case \cite[Lemma 6, p. 16]{okabook}), there exists a real analytic curve $p:[0,\varepsilon')\longrightarrow\C^n$, such that $p(0)=0$ and $p(s)=(\xi_1(s),\dots,\xi_n(s))\in\mathscr{A}_J$ for $s\neq0$. Now, consider the Taylor expansion 
    \[
    \xi_i(s)=a_is^{w_i}+\cdots, \quad 1\leq i\leq n
    \]
    where $w_i>0$ and $a_i\neq0$ for any $1\leq i\leq n$. Using the map $h^J$, the coordinates of the curve $\gamma$ are written in the form
    \begin{align*}
    t(s)   &= A_0s+\cdots \\
    z_i(s) &=a^{b_i}s^{\lip w,b_i\rip}+\cdots, \quad 1\leq i\leq r 
    \end{align*}
    in which $w=(w_1,\dots,w_n) \in ({\mathbb{N}^*})^n$ and $a=(a_1,\dots,a_n)\in(\C^*)^n$. Since compact faces of $L_{f_t}$ are compact faces of $f_t$ (see Remark \ref{obsfc}), let $\Delta_w\subset\Gamma(f_t;z)$ be the compact face associated with $w$ and let $d$ be the minimal value of $\lip w,x\rip$ for $x\in\Gamma_+(L_{f_t};\xi)$. Denoting by $W=(\lip w,b_1\rip,\dots,\lip w,b_r\rip)$ and using the non-degeneracy condition (see Remark \ref{obsng}) together with Lemma \ref{pesos} we conclude that for some $i_0\in\{1,\dots,r\}$
    \begin{equation}
        \dfrac{\partial (F_\tau)_{\Delta_w}}{\partial z_{i_0}}(h^J(a))\neq0.
    \end{equation}
    Therefore, in the case $I=\emptyset$, the non-degeneracy condition is sufficient to conclude the result. 
\end{proof}

%\section{Some examples of admissible families}

\begin{exem}\label{exdeterminantal} Let $q \geq 2$ be a positive integer and consider the toric surface $X(S_{\sigma}) \subset \mathbb{C}^{q+1}$ associated with the semigroup $S_{\sigma} =\lip b_1, \ldots, b_{q+1}\rip = \lip(1,0), (1,1), \ldots, (1,q)\rip$. The toric surface $X(S_{\sigma})$ is also a determinantal surface given by the zero set of the ideal generated by the $2\times2$ minor of the matrix
\begin{equation*}
%\label{quasimatriz}
	\begin{matrix}
	\begin{pmatrix}
			z_1  & \ \ \ \ \ \ & z_2 & \ \ \ \ \ \ & z_3 & \cdots &    z_{q-2} & \ \ \ \ \ \ & z_{q-1} & \ \ \ \ \ \  z_{q} \\
			
			z_2 & \ \ \ \ \ \ & z_3 & \ \ \ \ \ \ & z_4 &  \cdots   &   z_{q-1} & \ \ \ \ \ \ &  z_{q} & \ \ \ \ \ \  z_{q+1} \\
		\end{pmatrix},
	\end{matrix}
\end{equation*}
as we can see, for instance, in \cite[Example 1.1.6.]{cox2024}.

The canonical embedding $h: (\mathbb{C}^*)^2 \to X(S_{\sigma})$ is given by
\[
h(\xi_1, \xi_2) = (\xi_1, \xi_1 \xi_2, \ldots, \xi_1 \xi_2^q).
\]
Consider the family of polynomial functions on $D\times X(\Ssig)$ given by
\[
f(t,z) = z_1^2 +tz_i^d+ z_{q-1} \ \ \ \ \ \Rightarrow \ \ \ \ \ (L_{f_t})(\xi) = \xi_1^2 +t\xi_1^d\xi_2^{di}+\xi_1 \xi_2^{q-1},
\]
where $2\leq i\leq q-2$ and $d\geq2$. Then, 
$$
\text{supp}(f_t) = \{(2,0),(d,di), (1,q-1):2\leq i\leq q-2 \;\text{ and }\; d\geq2\},
$$
and the Newton polyhedron of $f_t$ has just the compact face
\[
\Delta_1 = \overline{AB}, \ \ \text{with} \ \ A = (2,0) \ \ \text{and} \ \ B = (1,q-1).
\]
It is easy to check that $f_t$ is non-degenerate.
Moreover, there are two non-compact faces $\Delta_2 = A + \mathbb{R}_{\geq 0} b_1$, $\Delta_3 = B + \mathbb{R}_{\geq 0} b_q$. The non-compact face $\Delta_2$ is not essential. However, the non-compact face $\Delta_3$ is essential. Indeed, the weight vector $w_3 = (q,-1)$ is orthogonal just to the generator  $b_{q+1} = (1,q)$, then $I_{\Delta_3} = \{q+1\}$ and it is easy to check that $f_t|_{X(S_{\sigma})^{I_{\Delta_3}}} \equiv 0$.

%Considere o vetor peso $w_3 = (q, -1)$. Então:

%\begin{itemize}
 % \item $D_1 = \overline{AB} = (t+1, -t(q-1)+q-1)$, com $0 \leq t \leq 1$  
  %\item $\langle w_3, D_1 \rangle = q(t+1) - (-t(q-1) + q - 1) = (2q - 1)t + q + 1$  
 % \item O mínimo ocorre em $t = 0$, com valor $q + 1$
%\end{itemize}

%Para $A_2 = (t, 0),\; t \geq 2$, temos:
%\[
%\langle w_3, A_2 \rangle = qt \Rightarrow \text{mínimo } = 2q
%\]

%Para $A_3 = (t+1, qt + q - 1),\; t \geq 0$, temos:
%\[
%\langle w_3, A_3 \rangle = q(t+1) - qt - q + 1 = 1
%\]

%Logo, o mínimo ocorre em toda a face $A_3$, que é essencial.

%Como $w_3 = (q, -1)$ acessa apenas $b_q$, temos $\text{Id}(w_3) = \{q+1\}$.

Now consider $u = (0, \ldots, 0, u_{q+1}) \in {X(S_{\sigma})^*}^{\{q+1\}}$, then
\[{\mathbb{C}^{*}_u}^{\{1,\dots,q+1\}} = \{(z_1,z_2,\dots,z_q,u_{q+1}), \ \ z_i \in \mathbb{C}^*\}.
\]
Then, on ${\mathbb{C}^{*}_u}^{\{1,\dots,q+1\}}$ the face function $(f_t)_{{\Delta}_3}(z) = z_{q-1}$ 
has no critical points with respect to the variables $z_1, \dots, z_q$. Thus, the family of polynomial functions $\{f_t\}$ is admissible. Therefore, by Theorem \ref{mainteo} the corresponding family of hypersurfaces $\{V(f_t)\}$ is Whitney equisingular.
\end{exem}

\begin{center}
	{ \bf Acknowledgments}
\end{center}

We would like to thank Daniel Duarte for all the helpful discussions he provided throughout the development of this work, as well as for his careful reading and valuable suggestions for improving the text.

%\begin{figure}[!thb]
%       \centering
%       \includegraphics[width=3cm]{poliex.tikz}
%     \caption{Polyhedral of family $f_t(z_1,z_2,z_3)=z_1^5+z_2^7z_3+z_3^{15}+tz_1z_2^6$}
%    \label{fig2}
%\end{figure}

%\textcolor{red}{Cuidado, esse exemplo a variedade é não normal}
%\begin{exem}[Bria\c con-Speder]
     %As a family of polynomial functions on $\C^3$ $$f_t(z_1,z_2,z_3)=z_1^5+z_2^7z_3+z_3^{15}+tz_1z_2^6,$$ does not satisfy the condition of being admissible, as its Newton boundary depends on the parameter $t$. Therefore, \cite[Theorem 3.8]{eyraloka1} cannot be applied to this family. On the other hand, let $\Ssig = \lip e_1,e_1+e_2,2e_2\rip$ be a  semigroup. Consider the same family, but on the toric variety $X(\Ssig)$. Then 
    % $$L_{f_t}(\xi_1,\xi_2)=\xi_1^5+\xi_1^8\xi_2^9+\xi_1^{15}\xi_2^{30}+ t\xi_1^6\xi_2^7$$
     %and consequently, its Newton boundary does not depend on $t$. Moreover, $\Gamma_{+}(f_t;z)$ has only one compact face $\Delta_1=A$, where $A=(5,5)$, which is non-degenerate and
    %locally tame along the vanishing varieties and also admissible. Thus, by Theorem \ref{mainteo} $f_t$ is Whitney equisingular on $X(\Ssig)$.   
%\end{exem}

%\bibliographystyle{plain}
\bibliography{refDD}

\end{document}